\documentclass[reqno]{amsart}
\usepackage{amssymb, amsmath, amscd}
\usepackage{color}
\date{Version v7a - last changed 23 February 2010 by PBG}
\newtheorem{theorem}{Theorem}

\newtheorem{remark}{Remark}
\def\qedbox{\hbox{$\rlap{$\sqcap$}\sqcup$}}
\def\mapright#1{\smash{\mathop{\longrightarrow}\limits\sp{#1}}}
\def\rho{\operatorname{Ric}}\def\varepsilon{h}
\def\rhoa{\Lambda\operatorname{Ric}}
\def\rhoR#1{R_{#1}}\def\rhoA#1{A_{#1}}
\def\langle{h(}\def\rangle{)}
\begin{document}
\title{Geometric realizations, curvature decompositions, and Weyl manifolds}
\author{Peter  Gilkey, Stana Nik\v cevi\'c, and Udo Simon}
\address{PG: Mathematics Department, University of Oregon\\
  Eugene OR 97403 USA\\
  E-mail: gilkey@uoregon.edu}
\address{SN: Mathematical Institute, Sanu,
Knez Mihailova 36, p.p. 367,
11001 Belgrade,
Serbia.\\ Email: stanan@mi.sanu.ac.rs}
\address{US: Institut f\"ur Mathematik, Technische Universit\"at
Berlin \newline  Strasse des 17. Juni 135, D-10623 Berlin, Germany\\ Email: simon@math.tu-berlin.de}
\begin{abstract}{We show any Weyl curvature model can be geometrically realized by a Weyl manifold
}\end{abstract}
\maketitle
\section{Introduction}\label{sect-1}

\subsection{Weyl geometry} Consider a
triple
$\mathcal{W}:=(N,g,\nabla)$ where
$g$ is a pseudo-Riemannian metric on a smooth $n$ dimensional manifold $N$ and where $\nabla$ is a torsion free connection on the tangent bundle
$TN$ of $N$. We suppose $n\ge2$ henceforth. We say that
$\mathcal{W}$ is a {\it Weyl manifold} if the following identity is satisfied:
\begin{equation}\label{eqn-1.a}
\nabla g=-2\phi\otimes g\quad\text{for some}\quad\phi\in C^\infty(T^*N)\,.
\end{equation}
This notion is conformally invariant. If $\mathcal{N}=(N,g,\nabla)$ is a Weyl manifold, then $\tilde{\mathcal{N}}:=(N,e^{2f}g,\nabla)$ is
again a Weyl manifold where $\tilde\phi:=\phi-df$. The simultaneous transformation of the pair $(g,\phi)$ is called a {\it gauge
transformation}, properties of the Weyl geometry that are invariant under gauge transformations are called {\it gauge invariants}. Let
$\nabla^g$ be the Levi-Civita connection defined by the pseudo-Riemannian metric $g$. There exists a conformally equivalent metric
$\tilde g$ locally so that $\nabla=\nabla^{\tilde g}$ if and only if $d\phi=0$; if $d\phi=0$, such a class exists
globally if and only if $[\phi]=0$ in de Rham cohomology. The metric is said to be {\it Einstein-Weyl} if the symmetrized Ricci tensor associated
to the connection $\nabla$ is a multiple of the metric; this condition is the natural analogue within conformal geometry of
Einstein's equation in Riemannian geometry. Thus Einstein-Weyl structures are of central importance.

Weyl \cite{W22} used these geometries in an attempt to unify gravity
with electromagnetism -- although this approach failed for physical reasons, the resulting geometries are still of importance
\cite{AI03,B98,C01,I02,N05,O06,PS91} and also appear in mathematical physics
\cite{C00,DT02,DMT01,I00,M04,NS04,OP98,P98}. We refer to \cite{H93,H94} for further details
concerning Weyl manifolds. The pseudo-Riemannian setting also is important
\cite{M01,S09} as is submanifold geometry \cite{M02} and contact geometry \cite{G09,M02x}. There are relations with para-conformal structures
\cite{D06} and spin geometry \cite{B00}. The literature in the field is vast and we can only give a flavor of it for reasons of brevity.

\subsection{Affine and pseudo-Riemannian geometry} Weyl geometry fits in between affine and pseudo-Riemannian geometry. We say that the pair
$\mathcal{A}:=(N,\nabla)$ is an {\it affine manifold} if
$\nabla$ is a torsion free connection. We say that the pair $\mathcal{N}:=(N,g)$ is a {\it pseudo-Riemannian manifold} if $g$ is a
pseudo-Riemannian metric on $N$. Since $\nabla^g$ is torsion free and $\nabla^g g=0$, the triple
$(N,g,\nabla^g)$ is a Weyl manifold. There are, however, examples with $d\phi\ne0$ so Weyl geometry is more general than
pseudo-Riemannian geometry or even conformal pseudo-Riemannian geometry. Every Weyl manifold gives rise to an underlying affine and an underlying
pseudo-Riemannian manifold; Equation (\ref{eqn-1.a}) provides the link between these two structures.

\subsection{Curvature} Let $\mathcal{R}(x,y)z:=(\nabla_x\nabla_y-\nabla_y\nabla_x-\nabla_{[x,y]})z$ be
the {\it curvature operator} of a torsion free connection $\nabla$. We use the metric to lower indices and define the $g$-associated
{\it curvature  tensor}
$R(x,y,z,w):=g(\mathcal{R}(x,y)z,w)$. This tensor belongs to $\otimes^4(T^*N)$ and satisfies the identities:
\begin{eqnarray}
&&R(x,y,z,w)+R(y,x,z,w)=0,\quad\text{and}\label{eqn-1.b}\\
&&R(x,y,z,w)+R(y,z,x,w)+R(z,x,y,w)=0\label{eqn-1.c}\,;
\end{eqnarray}
the symmetry of Equation (\ref{eqn-1.c}) is called the {\it Bianchi identity}. 
We define the {\it Ricci tensor}
$\rho$  by setting:
$$
\rho(x,y):=\operatorname{Tr}\{z\rightarrow\mathcal{R}(z,x)y\} \,.
$$

If $\nabla$ satisfies Equation (\ref{eqn-1.a}), i.e. if $(N,g,\nabla)$ is a Weyl manifold, then there is an additional curvature symmetry we
shall derive in Theorem \ref{thm-3.1} (see \cite{H94}):
\begin{equation}\label{eqn-1.d}
R(x,y,z,w)+R(x,y,w,z)=\textstyle\frac2n\{\rho(y,x) - \rho(x,y)\}g(z,w)\,.
\end{equation}
Similarly, if $\nabla=\nabla^g$ arises from pseudo-Riemannian geometry, then one has an additional symmetry:
\begin{equation}\label{eqn-1.e}
R(x,y,z,w) + R(x,y,w,z)=0 \,.
\end{equation}

\subsection{The algebraic context}
It is convenient to work in an abstract algebraic setting. Let $V$ be a real vector space of dimension $n \ge 2$ which is
equipped with a non-degenerate symmetric inner product $h$. Let
$\mathfrak{R}(V)\subset\otimes^4V^*$ be the space of all {\it generalized curvature tensors};
$A\in\mathfrak{R}(V)$ if and only if $A$ satisfies the symmetries given in Equations (\ref{eqn-1.b}) and (\ref{eqn-1.c}). 

The space of {\it algebraic curvature tensors}
$\mathfrak{A}(V)\subset\mathfrak{R}(V)$ is the subspace defined by imposing in addition the symmetry of Equation
(\ref{eqn-1.e}). An immediate algebraic consequence of Equations (\ref{eqn-1.b}), (\ref{eqn-1.c}), and (\ref{eqn-1.e}) is
the additional symmetry
\begin{equation}\label{eqn-1.f}
A(x,y,z,w)=A(z,w,x,y)\,.
\end{equation}
Note that the relations of Equation (\ref{eqn-1.e}) and of Equation (\ref{eqn-1.f}) are equivalent in the presence of Equations
(\ref{eqn-1.b}) and (\ref{eqn-1.c}) -- see, for example, the discussion in \cite{GNU09}.

Weyl geometry is intermediate between pseudo-Riemannian and affine geometry. We define the subspace of {\it Weyl generalized curvature tensors}
$\mathfrak{W}(V)$ by imposing Equations (\ref{eqn-1.b}), (\ref{eqn-1.c}), and (\ref{eqn-1.d}). Then:
$$\mathfrak{A}(V)\subset\mathfrak{W}(V)\subset\mathfrak{R}(V)\,.$$

\subsection{Geometric realization of curvature tensors}
We shall say that a generalized curvature tensor $A \in \mathfrak{R}(V)$ is {\it realized geometrically by an
affine manifold} $(N,\nabla)$ if there exists a point $P\in N$ and an isomorphism $\Xi:V\rightarrow T_PN$ so that $\Xi^*R_P=A$. The
following result
\cite{BNGS06} shows that Equations (\ref{eqn-1.b}) and (\ref{eqn-1.c}) generate the universal symmetries of a torsion free connection:

\begin{theorem}\label{thm-1.1}
 Every generalized curvature tensor can be realized geometrically by an affine manifold.
\end{theorem}
We have a  corresponding result for algebraic curvature tensors. We say that a triple $(V,h,A)$ 
is a {\it pseudo-Riemannian curvature model} if $h$ is a non-degenerate symmetric inner product on $V$ and if
$A\in\mathfrak{A}(V)$ is an algebraic curvature tensor. Such a model is said to be {\it geometrically realized by a pseudo-Riemannian manifold}
$(N,g)$ if there exists a point $P\in N$ and an isomorphism $\Xi:V\rightarrow T_PN$ so that $\Xi^*g_P=
h$ and so that
$\Xi^*R_P=A$. The following is well known (see, for example, the discussion in
\cite{BNGS06}) and shows that Equations (\ref{eqn-1.b}), (\ref{eqn-1.c}), and (\ref{eqn-1.e}) generate the universal symmetries of the
curvature tensor of the Levi-Civita connection:
\begin{theorem}\label{thm-1.2}
Every pseudo-Riemannian curvature model can be realized geometrically by a pseudo-Riemannian manifold.
\end{theorem}

We say that a triple $(V,h,A)$ is a {\it Weyl curvature model} if $h$ is a non-degenerate
symmetric inner product on $V$ and if $A\in\mathfrak{W}(V)$ is a Weyl generalized curvature tensor. Such a model is said to be {\it geometrically
realized by a Weyl manifold}
$(N,g,\nabla)$ if there exists a point $P\in N$ and if there exists an isomorphism $\Xi:V\rightarrow T_PN$ so that
$\Xi^*g_P=h$ and so that
$\Xi^*R_P=A$. The following result extends Theorem \ref{thm-1.1} and Theorem \ref{thm-1.2} to the Weyl setting; it shows that Equations 
(\ref{eqn-1.b}), (\ref{eqn-1.c}), and (\ref{eqn-1.d}) generate the universal symmetries of the curvature tensor in Weyl geometry. It shows that
one can pass from the algebraic setting to the geometric setting and provides thereby many new examples.

\begin{theorem}\label{thm-1.3}
Every Weyl curvature model can be realized geometrically by a Weyl manifold.
\end{theorem}

\begin{remark}
\rm If $(N,g,\nabla)$ geometrically realizes $A$ at a point $P \in N$, by considering a suitable conformal deformation $(N,e^{2f}g,\nabla)$,
we can use the Cauchy-Kovalevskaya Theorem to construct a Weyl manifold where $f=O(|x-P|^3)$ which has constant scalar curvature and which
realizes $A$ at
$P$. The argument is essentially the same as that used in \cite{BGKNW09} to establish a similar fact in the pseudo-Riemannian setting so we omit
details in the interests of brevity.
\end{remark}

\subsection{Conjugate tensors}
Let $\{e_1,...,e_n\}$ be an orthonormal basis for $(V,h)$. We can contract indices to define a Ricci-type tensor $\rho^*$  and the {\it conjugate
curvature tensor}
$R^*$ by setting: 
$$
\operatorname{Ric}^*(x,y):=\sum_iR(x,e_i,e_i,y),\quad
R^*(x,y,w,z)=-R(x,y,z,w)\,.
$$
As pointed out in Section 2.3 of \cite{GNU09}, for 
$R \in \mathfrak{R}(V)$, the conjugate curvature tensor
does not necessarily satisfy Equation (\ref{eqn-1.c}), thus 
$R^*$ need not be an element of $ \mathfrak{R}(V)$, in general.
Note that $\operatorname{Ric}^*(R) = \operatorname{Ric}(R^*)$. We present the following result, which is of interest in its own right, as an
application of the techniques which we shall derive in this paper:

\begin{theorem}\label{thm-1.5}
 Assume $n\ge3$. 
\begin{enumerate}
\item Let $A\in\mathfrak{W}(V)$ be a Weyl generalized curvature tensor. Then the conjugate tensor
$A^*$ satisfies the Bianchi identity if and only if
$A$ is an algebraic curvature tensor.
\item Let $\mathcal{N}=(N,g,\nabla)$ be a Weyl manifold where $H^1(N;\mathbb{R})=0$. Assume that the conjugate curvature tensor is a
generalized curvature tensor. Then there exists $f\in C^\infty(N)$ so that $\nabla$ is the Levi-Civita connection of the
conformally equivalent metric $e^{2f}g$.
\end{enumerate}\end{theorem}

 Here is a brief outline to the paper. In Section \ref{sect-2}, we discuss a
curvature decomposition of $\mathfrak{W}(V)$ under the action of the orthogonal group $\mathcal{O}=O(V,h)$ which is based
on work of Higa
\cite{H94}. In Section \ref{sect-3}, we establish some basic geometrical properties of Weyl manifolds and give an algorithm for constructing
Weyl manifolds. Section \ref{sect-4} is devoted to the proof of Theorem \ref{thm-1.3} and Section \ref{sect-6} is devoted to the proof of
Theorem \ref{thm-1.5}. The Higa curvature decomposition discussed in Section \ref{sect-2}
is central to our analysis; in a subsequent paper
\cite{GNU10}, we will examine algebraic properties of this decomposition in more detail.

\section{Curvature decompositions}\label{sect-2}
Let $\rhoa(x,y):=\textstyle\frac12(\rho(x,y)-\rho(y,x))$ be the {\it alternating Ricci tensor}; one has that $\rhoa\in\Lambda^2(V^*)$. 
We now define $\sigma:\Lambda^2(V^*)\rightarrow\otimes^4V^*$. If $\psi\in\Lambda^2(V^*)$, set:
\begin{eqnarray*}
&&\sigma(\psi)(x,y,z,w):=2\psi(x,y)\langle z,w\rangle+\psi(x,z)\langle y,w\rangle-\psi(y,z)\langle x,w\rangle\\
&&\phantom{\sigma(\psi)(x,y,z,w):}-\psi(x,w)\langle y,z\rangle+\psi(y,w)\langle x,z\rangle\,.
\end{eqnarray*}
The following decomposition of $\mathfrak{W}(V)$ is based on work of Higa \cite{H94}; we refer to \cite{B90,GNU09,ST69} for related
results concerning decompositions of $\mathfrak{R}(V)$ and of $\mathfrak{A}(V)$ and to \cite{GNU10} for further information concerning
the curvature decomposition of $\mathfrak{W}(V)$ as an $\mathcal{O}$ module.

\begin{theorem}\label{thm-2.1}
\ \begin{enumerate}
\item $\rhoa\sigma(\psi)=-n\psi$ for $\psi\in\Lambda^2(V^*)$.
\item $\sigma:\Lambda^2(V^*)\rightarrow\mathfrak{W}(V)$.
\item $\mathfrak{A}(V)=\mathfrak{W}(V)\cap\ker(\rhoa)$.
\item There is an $\mathcal{O}$
equivariant short exact sequence
$$0\rightarrow\mathfrak{A}(V)\rightarrow\mathfrak{W}(V)\mapright{\rhoa}\Lambda^2(V^*)\rightarrow0$$
which is equivariantly split by
$-\frac1n\sigma$.

\end{enumerate}
\end{theorem}

\begin{proof} 

Let $\psi\in\Lambda^2(V^*)$. Let $A=\sigma(\psi)$. Clearly $A(x,y,z,w)=-A(y,x,z,w)$. We
show that $A$ satisfies the Bianchi identity by computing:
\medbreak\quad
$A(x,y,z,w)+A(y,z,x,w)+A(z,x,y,w)$
\par\qquad$=2\psi(x,y)\langle z,w\rangle
+\psi(x,z)\langle y,w\rangle-\psi(y,z)\langle x,w\rangle$
\par\qquad\phantom{.}$+2\psi(y,z)\langle x,w\rangle
+\psi(y,x)\langle z,w\rangle-\psi(z,x)\langle y,w\rangle$
\par\qquad\phantom{.}$+2\psi(z,x)\langle y,w\rangle
+\psi(z,y)\langle x,w\rangle-\psi(x,y)\langle z,w\rangle$
\par\qquad\phantom{.}$-\psi(x,w)\langle y,z\rangle+\psi(y,w)\langle x,z\rangle$
\par\qquad\phantom{.}$-\psi(y,w)\langle z,x\rangle+\psi(z,w)\langle y,x\rangle$
\par\qquad\phantom{.}$-\psi(z,w)\langle x,y\rangle+\psi(x,w)\langle z,y\rangle$
\smallbreak\qquad$=0$.
\medbreak\noindent This shows that $A\in\mathfrak{R}(V)$ is a generalized curvature tensor. To show $A\in\mathfrak{W}(V)$ is a Weyl generalized
curvature tensor, we must establish that Equation (\ref{eqn-1.d}) holds. Let
$\{e_i\}$ be a basis for $V$. Adopt the Einstein convention and sum over repeated indices. Let $\varepsilon_{ij}:=\langle e_i,e_j\rangle$
and let $\varepsilon^{ij}$ be the inverse matrix. Let
$\rhoA{ij}$ and
$A_{ijkl}$ be the components of
$\rho$ and of
$A$ relative to this basis. We have $\rhoA{jk}=\varepsilon^{il}A_{ijkl}$. We establish Assertion (1) by computing:
\begin{eqnarray*}
\rhoA{jk}&=&\varepsilon^{il}\{2\psi_{ij}\varepsilon_{kl}+\psi_{ik}\varepsilon_{jl}-\psi_{jk}\varepsilon_{il}
   -\psi_{il}\varepsilon_{jk}+\psi_{jl}\varepsilon_{ik}\}\\
&=&2\psi_{kj}+\psi_{jk}-n\psi_{jk}+0+\psi_{jk}\\
&=&-n\psi_{jk}\,.
\end{eqnarray*}

By Assertion (1), $\rho\in\Lambda^2(V^*)$. We show that $A\in\mathfrak{W}(V)$ and establish Assertion (2) by using Assertion (1)
to compute:
$$
A_{ijkl}+A_{ijlk}=4\psi_{ij}\varepsilon_{kl}=\textstyle\frac2n(\rhoA{ji}-\rhoA{ij})\varepsilon_{kl}\,.
$$

If $A\in\mathfrak{R}(V)$, let 
$$S(A)(x,y,z,w)=\textstyle\frac12\{A(x,y,z,w)+A(x,y,w,z)\}$$
be the symmetrization of $A$ in the last 2 indices; $S(A)\in\Lambda^2(V^*)\otimes S^2(V^*)$. This symmetrization plays a central role in the
curvature decomposition of $\mathfrak{R}(V)$ given by Bokan \cite{B90}. It is then immediate that
$\mathfrak{A}(V)=\ker(S)\cap\mathfrak{R}(V)$. One verifies easily that if
$A\in\mathfrak{A}(V)$ is an algebraic curvature tensor, then both sides of Equation (\ref{eqn-1.d}) vanish so $A\in\mathfrak{W}(V)$ is a Weyl
generalized curvature tensor. If
$A\in\mathfrak{W}(V)$, then
$S(A)=0$ if and only if $\rhoa(A)=0$ by Equation (\ref{eqn-1.d}). This shows 
$$\ker(\rhoa)\cap\mathfrak{W}(V)=\mathfrak{A}(V)$$ and establishes
Assertion (3); Assertion (4) is a direct consequence of Assertions (1), (2), and (3).
\end{proof}

\section{The geometry of Weyl manifolds}\label{sect-3}
Let $(N,g,\nabla)$ be a Weyl manifold. Let $\{e_i\}$ be a local frame for $TN$ and let $\{e^i\}$ be the corresponding dual frame for $T^*N$. If
$\phi=\phi_ie^i$ is a
$1$-form, let
$\xi:=\phi_ie_i$ be the corresponding dual vector field.   We can establish some basic geometrical properties of Weyl manifolds as follows:

\begin{theorem}\label{thm-3.1}
Let $\nabla$ be a torsion free connection on a pseudo-Riemannian manifold $(N,g)$. Let $\phi$ be a smooth $1$-form on $N$.
\begin{enumerate}
\item The following assertions are equivalent and define the notion of a Weyl manifold:
\begin{enumerate}
\item $\nabla g=-2\phi\otimes g$.
\item $\nabla_xy=\nabla_x^gy+\phi(x)y+\phi(y)x-g(x,y)\xi$.
\end{enumerate}
\item If $(N,g,\nabla)$ is a Weyl manifold, then:
\begin{enumerate}
\item $d\phi=-\frac1n\rhoa$.
\item $R(x,y,z,w)+R(x,y,w,z)=\textstyle\frac2n\{\rho(y,x)-\rho(x,y)\}g(z,w)$.
\end{enumerate}\end{enumerate}
\end{theorem}

\begin{proof}
Let $\Gamma_{ij}{}^k$ be the Christoffel symbols of the connection $\nabla$ and let $\Gamma_{ij}^g{}^k$ be the
Christoffel symbols of the Levi-Civita connection. We lower indices to define
$\Gamma_{ijk}$ and express
$\Gamma_{ijk}=\Gamma_{ijk}^g+\alpha_{ijk}$. Then:
$$g_{jk;i}:=e_ig_{jk}-g(\nabla_{e_i}e_j,e_k)-g(e_j,\nabla_{e_i}e_k)=-\alpha_{ijk}-\alpha_{ikj}\,.$$
Thus the assertion that $\nabla g=-2\phi\otimes g$ is equivalent to the identity:
\begin{equation}\label{eqn-3.a}
2\phi_ig_{jk}=\alpha_{ijk}+\alpha_{ikj}\,.
\end{equation}
Similarly, the assertion $\nabla_xy=\nabla_x^gy+\phi(x)y+\phi(y)x-g(x,y)\xi$ is equivalent to the identity:
\begin{equation}\label{eqn-3.b}
\alpha_{ijk}=\phi_ig_{jk}+\phi_jg_{ik}-\phi_kg_{ij}\,.
\end{equation}
Thus to establish that Assertions (1-a) and (1-b) are equivalent, we must show that Equations (\ref{eqn-3.a}) and (\ref{eqn-3.b})
are equivalent algebraically if we assume the symmetry $\alpha_{ijk}=\alpha_{jik}$ which is the condition that
$\nabla$ is torsion free. Suppose first that Equation (\ref{eqn-3.a}) is satisfied. We compute:
\medbreak\qquad
$\alpha_{ijk}=-\alpha_{ikj}+2\phi_ig_{jk}=-\alpha_{kij}+2\phi_ig_{jk}
=\alpha_{kji}+2\phi_ig_{jk}-2\phi_kg_{ij}$
\smallbreak\qquad\qquad
$=\alpha_{jki}+2\phi_ig_{jk}-2\phi_kg_{ij}
=-\alpha_{jik}+2\phi_ig_{jk}-2\phi_kg_{ij}+2\phi_jg_{ik}$.
\medbreak\noindent Equation (\ref{eqn-3.b}) now follows. Conversely, suppose that Equation (\ref{eqn-3.b}) is
satisfied. We complete the proof of Assertion (1) by checking Equation (\ref{eqn-3.a}) is satisfied:
\begin{eqnarray*}
\alpha_{ijk}+\alpha_{ikj}&=&\phi_ig_{jk}+\phi_jg_{ik}-\phi_kg_{ij}+\phi_ig_{kj}+\phi_kg_{ij}-\phi_jg_{ik}\\
&=&2\phi_ig_{jk}\,.
\end{eqnarray*}

Suppose that Assertion (1-b) is satisfied. We must establish that the identity of Equation (\ref{eqn-1.d}) holds. We have
$$
R_{ijkl}=\partial_i\Gamma_{jkl}-\partial_j\Gamma_{ikl}+g^{rs}(\Gamma_{irl}\Gamma_{jks}-\Gamma_{jrl}\Gamma_{iks})\,.
$$
Fix a point $P$ of $N$. Choose geodesic coordinates on $N$ which are centered at $P$ for the given metric $g$. The Christoffel symbols of
$\nabla^g$ then vanish at $P$ so we have:
\begin{equation}\label{eqn-3.c}
R_{ijkl}(0)
=R_{ijkl}^g(0)
+\{\partial_i\alpha_{jkl}-\partial_j\alpha_{ikl}+g^{rs}(\alpha_{irl}\alpha_{jks}-\alpha_{jrl}\alpha_{iks})\}(0)\,.
\end{equation}
By Equation (\ref{eqn-3.b}),
$$
g^{ik}\alpha_{ijk}=g^{ik}\{\phi_ig_{jk}+\phi_jg_{ik}-\phi_kg_{ij}\}
=\phi_j+n\phi_j-\phi_j=n\phi_j\,.
$$
We use Equation (\ref{eqn-3.c}) to see
\begin{equation}\label{eqn-3.d}
\rhoR{jk}(0)=\rhoR{jk}^g(0)+g^{il}\{\partial_i\alpha_{jkl}-\partial_j\alpha_{ikl}
+g^{rs}(\alpha_{irl}\alpha_{jks}-\alpha_{jrl}\alpha_{iks})\}(0)\,.
\end{equation}
We have that $\rhoR{jk}^g=\rhoR{kj}^g$ and that $dg(0)=0$. Thus when we anti-symmetrize Equation (\ref{eqn-3.d}) and use Equation
(\ref{eqn-3.b}) we see
\begin{eqnarray*}
&&(\rhoR{jk}-\rhoR{kj})(0)=g^{il}\{-\partial_j\alpha_{ikl}+\partial_k\alpha_{ijl}\}(0)\\
&=&g^{il}\{-\partial_j(\phi_ig_{kl}+\phi_kg_{il}-\phi_lg_{ik})+\partial_k(\phi_ig_{jl}+\phi_jg_{il}-\phi_lg_{ij})\}(0)\\
&=&g^{il}\{-\partial_j(\phi_kg_{il})+\partial_k(\phi_jg_{il})\}(0)=n\{\partial_k\phi_j-\partial_j\phi_k\}(0)\,.
\end{eqnarray*}
This shows that $\rhoa(0)=-\frac1nd\phi(0)$. Since the point $P$ was arbitrary, Assertion (2-a) follows.
Since 
$$R_{ijkl}^g+R_{ijlk}^g=0\quad\text{and}\quad\alpha_{ijk}+\alpha_{ikj}=2\phi_ig_{jk},$$
we may compute that:
\medbreak\qquad
$\{R_{ijkl}+R_{ijlk}\}(0)=
   \{\partial_i(\alpha_{jkl}+\alpha_{jlk})-\partial_j(\alpha_{ikl}+\alpha_{ilk})$
\smallbreak\qquad\qquad
$+g^{rs}(\alpha_{irl}\alpha_{jks}+\alpha_{irk}\alpha_{jls}-\alpha_{jrl}\alpha_{iks}-\alpha_{jrk}\alpha_{ils}\}(0)$
\smallbreak\qquad\qquad
$=\{2(\partial_i\phi_j-\partial_j\phi_i)g_{kl}\}(0)
+g^{rs}\{(-\alpha_{ilr}+2\phi_ig_{lr})\alpha_{jks}$
\smallbreak\qquad\qquad\quad
$+(-\alpha_{ikr}+2\phi_ig_{kr})\alpha_{jls}-\alpha_{jsl}\alpha_{ikr}-\alpha_{jsk}\alpha_{ilr}\}(0)$
\smallbreak\qquad\qquad
$=\{2(\partial_i\phi_j-\partial_j\phi_i)g_{kl}\}(0)+g^{rs}\{2\phi_ig_{lr}\alpha_{jks}-2\alpha_{ilr}\phi_jg_{ks}$
\smallbreak\qquad\qquad\quad
$+2\phi_ig_{kr}\alpha_{jls}-2\alpha_{ikr}\phi_jg_{ls}\}(0)$
\smallbreak\qquad\qquad
$=\{2(\partial_i\phi_j-\partial_j\phi_i)g_{kl}\}(0)+g^{rs}\{2\phi_i\alpha_{jkl}-2\phi_j\alpha_{ilk}$
\smallbreak\qquad\qquad\quad
$+2\phi_i\alpha_{jlk}-2\phi_j\alpha_{ikl}\}(0)$
\smallbreak\qquad\qquad
$=\{2(\partial_i\phi_j-\partial_j\phi_i)g_{kl}\}(0)+g^{rs}\{2\phi_i\phi_jg_{kl}-2\phi_j\phi_ig_{kl}\}$
\smallbreak\qquad\qquad
$=2\{(\partial_i\phi_j-\partial_j\phi_i)g_{kl}\}(0)=\frac2n\{(\rhoR{ji}-\rhoR{ij})g_{kl}\}(0)$.
\end{proof} 

\section{The proof of Theorem \ref{thm-1.3}}\label{sect-4}
Let $A\in\mathfrak{W}(V)$. We apply Theorem \ref{thm-2.1} to decompose $A=A^1+A^2$ for $A^1\in\mathfrak{A}(V)$ and for $A^2=\sigma(\psi)$ where
$\psi\in\Lambda^2$, i.e. 
$$
  A^2_{ijkl}=2\psi_{ij}\varepsilon_{kl}+\psi_{ik}\varepsilon_{jl}
  -\psi_{jk}\varepsilon_{il}-\psi_{il}\varepsilon_{jk}+\psi_{jl}\varepsilon_{ik}\,.
$$
By Theorem
\ref{thm-1.2}, choose a pseudo-Riemannian manifold $(N,g)$ which geometrically realizes $(V,h,A^1)$ at some point
$P$; identify $T_PM$ with $V$ henceforth and $h$ with $g(0)$.
Choose $g$-geodesic coordinates $(x^1,...,x^n)$ centered at $P$. Then $g_{ij}=\varepsilon_{ij}+O(|x|^2)$. Consider the $1$-form
$\phi:=\psi_{il}x^ldx^i$. We form the connection $\nabla$ of Theorem \ref{thm-3.1} to construct a Weyl manifold $(N,g,\nabla)$. The
Christoffel symbols of the connection are given by
$\Gamma_{ijk}=\Gamma_{ijk}^g+\alpha_{ijk}$ where
$$\alpha_{ijk}=\psi_{li}x^l\varepsilon_{jk}+\psi_{lj}x^l\varepsilon_{ik}-\psi_{lk}x^l\varepsilon_{ij}\,.$$
Since $g=g(0)+O(|x|^2)$, $\Gamma^g=O(|x|)$. Since $\alpha(0)=0$, $\Gamma(0)=0$ and thus:
\begin{eqnarray*}
&&R_{ijkl}(0)=(\partial_i\Gamma_{jkl}-\partial_j\Gamma_{ikl})(0)\\
&=&R_{ijkl}^g(0)+(\psi_{ij}\varepsilon_{kl}+\psi_{ik}\varepsilon_{jl}-\psi_{il}\varepsilon_{jk})
-(\psi_{ji}\varepsilon_{kl}+\psi_{jk}\varepsilon_{il}-\psi_{jl}\varepsilon_{ik})\\
&=&A^1_{ijkl}+A^2_{ijkl}\,.
\end{eqnarray*}
This completes the proof of Theorem \ref{thm-1.3}.\hfill\qedbox

\section{The proof of Theorem \ref{thm-1.5}}\label{sect-6}
We work in the algebraic setting to establish Theorem \ref{thm-1.5} (1).
Let $A\in\mathfrak{W}(V)$. We decompose $A=A_1+A_2$ where $A_1\in\mathfrak{A}(V)$ and $A_2=\sigma\psi$ for some
$\psi\in\Lambda^2(V)$. Since $A_1^*=A_1$ and $A_1$ satisfies the Bianchi identity, $A^*$ satisfies the Bianchi identity if and
only if
$A_2^*$ satisfies the Bianchi identity. We compute:
\begin{eqnarray*}
&&A_2(x,y,z,w)=2\psi(x,y)h( z,w)+\psi(x,z)h( y,w)-\psi(y,z)h( x,w)\\
&&\phantom{\sigma(\psi)(x,y,z,w):}-\psi(x,w)h( y,z)+\psi(y,w)h( x,z)\\
&&A_2^*(x,y,z,w)=-2\psi(x,y)h(w,z)-\psi(x,w)h(y,z)+\psi(y,w)h(x,z)\\
&&\phantom{\sigma(\psi)(x,y,w,z):}+\psi(x,z)h(y,w)-\psi(y,z)h(x,w)
\end{eqnarray*}
We check the Bianchi identity:
\begin{eqnarray*}
&&A_2^*(x,y,z,w)+A_2^*(y,z,x,w)+A_2^*(z,x,y,w)\\
&=&-2\psi(x,y)h(w,z)-\psi(x,w)h(y,z)+\psi(y,w)h(x,z)\\
&&-2\psi(y,z)h(w,x)-\psi(y,w)h(z,x)+\psi(z,w)h(y,x)\\
&&-2\psi(z,x)h(w,y)-\psi(z,w)h(x,y)+\psi(x,w)h(z,y)\\
&&\qquad+\psi(x,z)h(y,w)-\psi(y,z)h(x,w)\\
&&\qquad+\psi(y,x)h(z,w)-\psi(z,x)h(y,w)\\
&&\qquad+\psi(z,y)h(x,w)-\psi(x,y)h(z,w)\\
&=&-4\psi(x,y)h(w,z)-4\psi(y,z)h(w,x)-4\psi(z,x)h(w,y)\,.
\end{eqnarray*}
Suppose $n\ge3$ and that $A$ satisfies the Bianchi identity. Let $\{e_1,...,e_n\}$ be an orthonormal basis for $V$. Let $\{i,j,k\}$ be distinct
indices. We take
$x=e_i$,
$y=e_j$,
$z=w=e_k$ to conclude
$\psi(e_i,e_j)=0$. This proves $\psi=0$ so $A=A_1\in\mathfrak{A}(V)$. Conversely,
of course, if $A\in\mathfrak{A}(V)$, then $A^*=A\in\mathfrak{A}(V)$.

Next, we examine the geometric setting. Let $\mathcal{W}=(N,g,\nabla)$ be a Weyl manifold. We suppose the conjugate curvature
$R^*$ tensor of $\nabla$ is a generalized curvature tensor. By Assertion (1), that implies that $R$ is an algebraic curvature
tensor and hence $\rhoa=0$. Thus $d\phi=0$ by Theorem \ref{thm-3.1} (2a). Since we assumed the de Rham cohomology group
$H^1(N;\mathbb{R})=0$, we can choose $f\in C^\infty(N)$ so that
$\phi=df$. Then
$(N,e^{2f}g,\nabla)$ is a again a Weyl manifold where $\tilde\phi=\phi-df=0$. Thus $\nabla$ is the Levi-Civita connection
of the metric $e^{2f}g$.
\hfill\qed.

\section*{Acknowledgments} Research of the first and third  author was partially supported by DFG PI 158/4-6 (Germany), research of the second
author was
 partially supported by a research grant of the TU Berlin and by Project 144032 (Serbia).

\end{document}